\documentclass[12pt,reqno]{amsart}
\usepackage{amsthm}

\usepackage[x11names]{xcolor}
\usepackage[linkcolor=Blue2,citecolor=red,colorlinks]{hyperref}
\usepackage[capitalize,nameinlink]{cleveref}

\usepackage{color, bm, amscd, tikz-cd}
\newcommand{\cA}{{\mathcal A}}
\newcommand{\cB}{{\mathcal B}}

\newcommand{\cE}{{\mathcal E}}
\newcommand{\cF}{{\mathcal F}}

\newcommand{\cH}{{\mathcal H}}

\newcommand{\cK}{{\mathcal K}}

\newcommand{\cN}{{\mathcal N}}

\newcommand{\cS}{{\mathcal S}}

\newcommand{\bC}{{\mathbb{C}}}

\newcommand{\bD}{{\mathbb D}}

\newtheorem{thm}{Theorem}[section]

\theoremstyle{definition}

\newtheorem{definition}[thm]{Definition}
\newtheorem{remark}[thm]{Remark}

\newcounter{tmp}

\numberwithin{equation}{section}

\begin{document}
	\title[Factorization]{Factorization of functions in the Schur-Agler class related to test functions}

	\author[Bhowmik]{Mainak Bhowmik}
	\address{Department of Mathematics\\
	Indian Institute of Science\\
	Bangalore 560012, India}
	\email{mainakb@iisc.ac.in}

	\author[Kumar]{Poornendu Kumar}
	\address{Department of Mathematics\\
	University of Manitoba, Winnipeg, Canada, R3T 2N2}
	\email{Poornendu.Kumar@umanitoba.ca}
	

	\maketitle
	\begin{abstract}
		We provide necessary and sufficient conditions for operator-valued functions on  arbitrary sets associated with a collection of test functions to have factorizations in several situations.  
	\end{abstract}
	\section{Introduction}
	A remarkable result in function theoretic operator theory says that a holomorphic function $\varphi:\mathbb{D}\rightarrow\overline{\mathbb{D}}$ has a {\em realization formula} $$\varphi(z)=A +zB(I-zD)^{-1}C $$ for an isometry $ U= \left[\begin{smallmatrix}
		A & B \\
		C & D 
	\end{smallmatrix}\right]$ on $\mathbb{C}\oplus\cH$ where $\cH$ is a Hilbert space determined by $\varphi$.

	 In general, on a domain $\Omega \subseteq \bC^d$, a holomorphic function $\theta$, taking values in $\cB(\cE)$ which is the $C^*$-algebra of bounded linear operators on some Hilbert space $\cE$, is said to be a {\em Schur class function} if $\|\theta(z) \| \leq 1$ for all $z$ in $\Omega$. The realization formula for the Schur class functions has been generalized on various domains such as an annulus \cite{DM}, the bidisc \cite{AM-Book}, the complex unit ball \cite{Ball-Bolo}, and the symmetrized bidisc \cite{SYM_Real, Tirtha-Hari-JFA}. Interestingly, not every Schur class function on the polydisc $\bD^d$ ($d\geq 3$) has a realization formula. However, a proper subclass, known as the {\em Schur-Agler class}, of the Schur class on $\bD^d$ does. See \cite{Agler}. It has been generalized to an abstract setting where the domain $\Omega$ is replaced by a set $X$, and the Schur class is substituted with a specific class of functions that depend on a collection of test functions $\Psi$ defined on $X$. This class is known as the {\em $\Psi$-Schur-Agler class}. We shall elaborate on this in \cref{REAL}.
	
	The realization formula is immensely powerful, giving rise to a wide array of results. To mention a few, it facilitates the derivation of the Pick-Nevanlinna interpolation \cite{AM-Book}, proves the commutant lifting theorem \cite{Ball}, and establishes the Caratheodory approximation result \cite{ABJK}. Furthermore, its utility extends to signal processing \cite{Gohberg} and electrical engineering \cite{Helton}. In this article we shall employ it for the purpose of factorization.
	
	By a {\em factorization} of a $\Psi$-Schur-Agler class  function $\theta$, we mean $\theta=\theta_1\theta_2$ for some $\theta_1$ and $\theta_2$ in the $\Psi$-Schur-Agler class. The factorization of classical Schur functions traces back to the pioneering work of Sz.-Nagy and Foias \cite{Nagy-Foias} and Brodskii \cite{Brod, Brodskii}, who investigated them to analyze invariant subspaces of specific operators, along with their relation with the  characteristic functions.  Notably, they established a one-to-one correspondence between the invariant subspaces of contractions and certain factorizations of the characteristic functions of contractions. Interested readers are encouraged to consult the book by Sz.-Nagy and Foias \cite[Chapters, 6 and 7]{Nagy-Foias} as well as a recent work of Curto, Hwang and Lee on shift-invariant subspaces \cite{Curto2}.
	
	Furthermore, this concept is intricately linked to extreme points. Forelli, in \cite{FF0}, proved that a function $f$ is an extreme point of the set of Herglotz class functions if and only if the inverse Cayley transform of $f$, which lies in the Schur class on $\mathbb{D}$, cannot be factorized. However, when we extend this inquiry to arbitrary domains, Forelli's subsequent work \cite{FF} showed that only one direction holds true. Specifically, Forelli's result asserts that under ceratin conditions on $\Omega$ {\em if $f$ is an extreme point of the set $\mathcal{N}(\Omega, p)$, then the inverse Cayley transform of $f$, belonging to the Schur class on $\Omega$, cannot be factorized}. Here, $\mathcal{N}(\Omega, p)$ represents the normalized Herglotz class of functions (see \cite{ Herglotz} for more details) defined as:
	$$\mathcal{N}(\Omega, p) = \{ f:\Omega\rightarrow\bC \text{ holomorphic with }  \operatorname{Re}f(z) > 0\ \text{for all}\ z\in \Omega\ \text{ and } f(p) = 1 \}.$$

	Therefore, it is natural to inquire about necessary and sufficient conditions for the Schur class functions to have a factorization. Understanding the factorization of such functions is quite challenging. However, the realization formula provides a promising avenue to unravel this complexity and determine when factorization is possible. There has been some work in these directions for the case of the disc, more generally the polydisc; please
	refer to \cite{Alpay, Brod, Brodskii,Ramlal-Jaydeb-1, Kal}. In this article, we present necessary and sufficient conditions on the blocks of the isometric colligation for operator-valued $\Psi$-Schur-Agler class functions on $X$ (endowed with a collection of test functions) to have factorizations in various scenarios which also generalize the previously known results in the operator-valued setting. Examples are given at the end. 
	
	\section{ The Realization formula} \label{REAL}
	 This section aims to provide a concise overview of test functions following \cite{DM, DMM}.
	A collection $\Psi$ of $\mathbb C$-valued functions on a set $X$ is called a set of {\em test functions} if the following conditions hold:
	\begin{enumerate}
		\item $\operatorname{sup}_{\psi\in \Psi} |\psi(x)|<1$ for all $x\in X$;
		\item for each finite subset $\Lambda$ of $X$, the collection $\{\psi|_{\Lambda}: \psi\in \Psi\}$ together with the constant functions generates the algebra of all $\mathbb C$-valued functions on $\Lambda$.
	\end{enumerate}
	The collection $\Psi$ inherits a subspace topology of the space of all bounded functions from $X$ to $\overline{\bD}$ endowed with the topology of point-wise convergence. We shall denote the algebra of bounded continuous functions over $\Psi$ with pointwise algebra operation by $C_b(\Psi)$. Define an injective mapping $E:X\rightarrow C_b(\Psi)$ as $E(x)= ev_x$, where $ev_x(\psi)=\psi(x)$ for $\psi \in \Psi$. Let $\cF$ be a Hilbert space. We say that a map $k: X\times X\rightarrow \mathcal{B}(C_b(\Psi), \mathcal{B}{(\cF}))$ is a {\em completely positive kernel} if the following holds:
	\begin{align}\label{Kernel}
		\sum_{i,j=1}^N T_j^*k(x_i, x_j)\left(\overline{f_j}f_i\right)T_i\geq 0
	\end{align}
	for all $x_1, \dots, x_N\in X $, $T_1, \dots, T_N\in \cB(\cF)$, $f_1, \dots, f_N\in C_b(\Psi)$ and $N \in \mathbb N$.
	
	A $\cB(\cF)$-valued kernel $S$ on $X$ is said to be {\em $\Psi$-admissible} if the map $M_\psi$, sending each element $h$ of the reproducing kernel Hilbert space $\cH_S $ to $\psi\cdot h$, is a contraction on $\cH_S$. Let $\mathcal{K}_{\Psi}(\cF)$ be the collection of all $\cB(\cF)$-valued $\Psi$-admissible kernels on $X$. For a Hilbert space $\cE$, we say that $f: X \rightarrow \cB(\cE)$ is in $H^\infty_{\Psi}(\cE)$ if there is a non-negative constant $C$ such that the $\cB(\cE \otimes \cE)$-valued function
	\begin{align}\label{HH}
		(x, y)\mapsto\left(C^2-f(y)^*f(x)\right)\otimes S(x,y)
	\end{align}
	to be a positive kernel for all $S$ in $\mathcal{K}_{\Psi}(\cE)$. If $f$ is in $H^\infty_{\Psi}(\cE)$, then we denote by $C_f$ the infimum of all such $C$ for \eqref{HH} is a positive kernel for all $S$ in $\mathcal{K}_{\Psi}(\cE)$. The collection of maps $f\in H^\infty_{\Psi}(\cE)$ for which
	$C_f$ is no larger than $1$ is called the {\em $\Psi-$Schur-Agler class} and is denoted by $\cS\cA_{\Psi}(\cE).$ We are ready to state the Realization formula in this context \cite{BBC, DM}.
	
	\setcounter{thm}{\thetmp}
	\begin{thm}\label{Real-thm}
		A function $f: X \to \cB(\cE)$ is in $ \cS\cA_{\Psi}(\cE)$ if and only if there exist a Hilbert space $\mathcal{H}$, a unital $*-$representation $\rho: C_b(\Psi) \to \cB(\cH)$ and an isometry
		\begin{align*}
			U=\begin{bmatrix}
				A & B\\
				C & D
			\end{bmatrix}
			:\begin{bmatrix}\cE\\\mathcal{H}\end{bmatrix}\to \begin{bmatrix}\cE\\\mathcal{H}\end{bmatrix}
		\end{align*}
		such that
		\begin{align}\label{Real-formula}
			f(x)= A+B \rho(E(x))(I-D \rho(E(x)))^{-1}C \ \text{for all}\ x\in X.
		\end{align}
	\end{thm}

	\section{Main Results} 
	 In this section, we find necessary and sufficient conditions for $\theta$ to have a factorization. We shall assume that there exists a point $x_0 \in X$ such that $E(x_0)=0$ in $C_b(\Psi)$. In fact, in \cite{BBC}, it has been shown that if $\Psi$ consists of holomorphic test functions on a domain $\Omega$ in $\bC^d$ and $z_0 \in \Omega$ then we can find another collection of holomorphic test functions $\Theta$ such that $\varphi(z_0)=0$ for each $\varphi \in \Theta$ and $\cK_\Psi(\cE) = \cK_\Theta(\cE)$. This suggests that whenever we have holomorphic test functions on $\Omega$ we can assume that there is a point $z_0\in \Omega$ such that $E(z_0)=0$.
	The results of this section are motivated by \cite{Ramlal-Jaydeb-1}. 
	\begin{definition}
		Given two Hilbert spaces $\cH_1$ and $\cH_2$, a unital $*$-representation $\rho : C_b(\Psi) \rightarrow \cB(\cH_1 \oplus \cH_2)$ is said to be reducible if $$\rho(g)(\cH_j) \subseteq \cH_j  \ \text{for}\  j=1,2\ \text{and}\ g \in C_b(\Psi).$$

	\end{definition}
	
	First, we consider the case when $\theta$ vanishes at $x_0$ and one of its factors is a self adjoint invertible operator at $x_0$. 
	\begin{thm}\label{Main1}
		Let $\theta \in \cS\cA_{\Psi}(\cE)$ be such that $\theta(x_0) = 0$. Then $\theta = \psi_1 \psi_2$ with $\psi_2(x_0)=A$, a self adjoint invertible operator on $\cE$, for some $\psi_1,\psi_2 \in \cS\cA_{\Psi}(\cH) $ if and only if there exist Hilbert spaces $\cH_1,\cH_2$, a reducible unital $*-$representation $ \rho :C_b(\Psi) \to \cB \left(\cH_1 \oplus \cH_2\right)$ and an isometric colligation,
		$$ U=  \left[\begin{array}{c|c c} 
			0 & B_1 & 0 \\ 
			\hline 
			C_1 & D_1 & D_2 \\
			C_2 & 0 & D_3
		\end{array}\right]  \,\,:\cE \oplus (\cH_1 \oplus \cH_2) \to \cE \oplus (\cH_1 \oplus \cH_2) $$
		with \begin{align}\label{cond_1}
			C_1 A^{-2}C_1^* D_2 = D_2, \,\, C_1^*C_1 = A^2
		\end{align} 
		such that $\theta$ is of the form \eqref{Real-formula}, where $$ B=\begin{bmatrix}
			B_1 & 0 
		\end{bmatrix},\,\, C= \begin{bmatrix}
			C_1\\
			C_2
		\end{bmatrix} ,\,\, D= \begin{bmatrix}
			D_1 & D_2 \\
			0 & D_3
		\end{bmatrix}.$$  
		
	\end{thm}
	\begin{proof}
		Suppose $\theta = \psi_1\psi_2$ and $\psi_2(x_0) = A$, where $A$ is self adjoint and invertible. Then $\psi_1(x_0)=0$. Now $\psi_1$ being in $\cS\cA_{\Psi}(\cE)$, by \cref{Real-thm}, there exist a unital $*-$representation $\rho_1: C_b(\Psi) \to \cB(\cH_1)$ and an isometric colligation $$ U_1 =  \left[\begin{array}{c|c} 
			0 & B_1 \\ 
			\hline \vspace{2mm}
			C_1 & D_1
		\end{array}\right] $$ such that $$\psi_1(x) = B_1 \rho(E(x))(I-D_1 \rho(E(x)))^{-1}C_1 $$ for all $x$ in $X$. Similarly, for $\psi_2$, there  exist a unital $*-$representation of $C_b(\Psi)$, $\rho_2$ (say) on a Hilbert space $\cH_2$ and an isometric colligation 
		$$ U_2 =\left[\begin{array}{c|c} 
			A & B_2 \\ 
			\hline \vspace{1mm}
			C_2 & D_2
		\end{array}\right] $$ such that $$ \psi_2(x) = A+B_2 \rho_2 (E(x))(I-D_2\rho_2(E(x)))^{-1}C_2.$$
		
		Now we define, a unital $*-$representation $\rho$ of $C_b(\Psi)$ on $\cH_1 \oplus \cH_2$ in the following way,
		$$ \rho (g) : =\begin{bmatrix}
			\rho_1(g) & 0\\
			0 & \rho_2(g) 
		\end{bmatrix} $$ for each $g\in C_b(\Psi)$. 
		Clearly, $\rho$ is a unital $*-$representation such that $\rho(\cH_j)\subseteq \cH_j$ for $j=1,2.$  So,
		$$ \rho(E(x))= \begin{bmatrix}
			\rho_1(E(x)) & 0\\
			0 & \rho_2(E(x))
		\end{bmatrix}. $$
		Set,
		$$ U=  \left[\begin{array}{c|c c} 
			0 & B_1 & 0 \\ 
			\hline
			C_1 & D_1 & 0 \\
			0 & 0 & I_{\cH_2}
		\end{array}\right] 
		\left[\begin{array}{c|c c} 
			A & 0 & B_2 \\ 
			\hline 
			0 & I_{\cH_1} & 0 \\
			C_2 & 0 &  D_2
		\end{array}\right] 
		= 
		\left[\begin{array}{c|c c} 
			0 & B_1 & 0 \\ 
			\hline 
			C_1 A & D_1 & C_1 B_2 \\
			C_2 & 0 & D_2
		\end{array}\right].
		$$
		
		\quad

		Since $U_1$ and $U_2$ are isometries, $U$ is an isometry.
		Let    
		\begin{align*}
			f(x) &= \begin{bmatrix}
				B_1 & 0 
			\end{bmatrix}  \left[\begin{smallmatrix}
				\rho_1(E(x)) & 0\\
				0 & \rho_2(E(x))
			\end{smallmatrix}\right] \left(I_{\cH_1 \oplus \cH_2} -\left[ \begin{smallmatrix}
				D_1 & C_1 B_2 \\
				0 & D_2
			\end{smallmatrix}\right] \left[ \begin{smallmatrix}
				\rho_1(E(x)) & 0\\
				0 & \rho_2(E(x))
			\end{smallmatrix}\right] \right) ^{-1} \begin{bmatrix}
				C_1 A\\
				C_2
			\end{bmatrix} \\
			&= \begin{bmatrix}
				B_1\rho_1 (E(x)) & 0 
			\end{bmatrix}  \begin{bmatrix}
				\left ( I_{\cH_1} - D_1\rho_1(E(x)) \right )^{-1} & Z \\
				0 & \left ( I_{\cH_2} - D_2\rho_2(E(x)) \right )^{-1}
			\end{bmatrix}  \begin{bmatrix}
				C_1 A\\
				C_2
			\end{bmatrix} 
		\end{align*} 
		where $$ Z= \left (I_{\cH_1} - D_1\rho_1(E(x))\right )^{-1} C_1B_2\rho_2(E(x)) \left(I_{\cH_2} - D_2 \rho_2(E(x)) \right)^{-1}.$$ A straightforward calculation gives
		\begin{align*}
			f(x) &=\left[ B_1\rho_1(E(x)) \left( I_{\cH_1} -D_1\rho_1(E(x))\right)^{-1} C_1 A\right ] +\\
			& \quad\quad\quad \quad \quad \left [ B_1 \rho_1(E(x)) \left( I_{\cH_1} -D_1 \rho_1(E(x)) \right)^{-1} C_1 B_2 \rho_2(E(x)) \left ( I_{\cH_2} - D_2\rho_2(E(x)) \right )^{-1} C_2 \right ] \\
			&= \psi_1(x)\psi_2(x)\\
			&= \theta(x).
		\end{align*}
		Suppose we write the isometry $$ U =  \left[\begin{array}{c|c c} 
			0 & \tilde{B_1} & 0 \\ 
			\hline 
			\tilde{C_1} & \tilde{D_1} & \tilde { D_2} \\
			\tilde{C_2} & 0 &\tilde{D_3}
		\end{array}\right].$$ Then $\tilde{B_1} = B_1,\,\tilde{C_1}=  C_1 A,\, \tilde{C_2} = C_2,\, \tilde{D_1}= D_1,\, \tilde { D_2} = C_1 B_2,\,\text { and } \tilde{D_3} = D_2 $.\\
		Since $U_1$ is an isometry,
		$$
		U_1^*U_1 = \begin{bmatrix}
			0 & C_1^* \\
			B_1^* & D_1^*
		\end{bmatrix} \begin{bmatrix}
			0 & B_1\\
			C_1 & D_1
		\end{bmatrix} = \begin{bmatrix}
			I_{\mathcal{E}} & 0 \\
			0 & I_{\cH_1}
		\end{bmatrix}
		$$ which gives 
		\begin{align} \label{iso-1}
			C_1^*C_1 =I_{\mathcal{E}}, \quad
			C_1^* D_1 = 0 \quad \text{ and } \quad
			B_1^* B_1 + D_1^* D_1 = I_{\cH_1}.
		\end{align}
		Similarly, $U_2$ being an isometry
		\begin{align}\label{iso-2}
			A^*A + C_2^*C_2 = I_{\mathcal{E}},\quad
			A^* B_2 + C_2^*D_2 = 0 \quad\text{and}\quad
			& B_2^*B_2 + D_2^*D_2 = I_{\mathcal{H}_2}.
		\end{align}
		Now,
		\begin{align*}
			\tilde{C_1} A^{-2}\tilde{C_1}^* \tilde{D_2} &= ( C_1 A)A^{-2}(C_1 A)^*C_1B_2\\&= C_1 A^{-1}A^* C_1^* C_1 B_2\\&= C_1 B_2\\&=\tilde{D_2}\quad\left(\text{since}\,\, C_1^*C_1 =I_{\mathcal{E}} \,\,\text{and } A= A^* \right)
		\end{align*}
		Also, using \cref{iso-1}, we get the following
		$$ \tilde{C_1}^*\tilde{C_1} = (C_1 A)^*(C_1 A) = A^*A = A^2. $$
		Therefore, the isometry $U$ satisfies the condition \eqref{cond_1}.

		Conversely, suppose that there exist Hilbert spaces $\mathcal{H}_1,\mathcal{H}_2$ and a reducible unital $*-$representation $ \rho :C_b(\Psi) \to \cB \left(\cH_1 \oplus \cH_2\right)$.
		So, $\rho $ has the following form: 
		$$\rho (g) = \begin{bmatrix}
			\rho(g)|_{\cH_1} & 0 \\
			0 & \rho(g)|_{\cH_2}
		\end{bmatrix}
		: \mathcal{H}_1 \oplus \mathcal{H}_2 \to \mathcal{H}_1 \oplus \mathcal{H}_2$$ for each $g\in C_b(\Psi)$.
		Define, $$\rho_1(g) := 
		\rho(g)|_{\cH_1}  \ \text{and} \  \rho_2(g) := \rho(g)|_{\cH_2}  $$ for every $g\in C_b(\Psi)$.
		Then $\rho_1$ and $\rho_2$ are both unital $*-$ representations.
		Let $$
		U= \left[ \begin{array}{c| c c}
			0 & B_1 & 0 \\
			\hline
			C_1 & D_1 & D_2 \\
			C_2 & 0 & D_3
		\end{array} \right]:
		\cE \oplus(\mathcal{H}_1\oplus\mathcal{H}_2) \rightarrow \cE \oplus (\mathcal{H}_1\oplus\mathcal{H}_2) $$
		be an isometric colligation such that 
		$$ \theta(x)=  B \rho(E(x)) (I_{\cH_1 \oplus \cH_2}- D \rho (E(x)))^{-1} C,$$
		for all $x\in X$, where $$ B=\begin{bmatrix}
			B_1 & 0 
		\end{bmatrix},\quad C= \begin{bmatrix}
			C_1\\
			C_2
		\end{bmatrix}\quad\text{and}\quad D= \begin{bmatrix}
			D_1 & D_2 \\
			0 & D_3
		\end{bmatrix},$$
		that satisfies \cref{cond_1} for some self adjoint and invertible operator $A$. Set,
		$$
		U_1 = \left[ \begin{matrix}
			0 & B_1 \\
			C_1 A^{-1} & D_1
		\end{matrix} \right ]\quad\text{and}\quad
		U_2 = \left[ \begin{matrix}
			A & B_2 \\
			C_2 & D_3
		\end{matrix}\right ]
		$$
		where $B_2 = A^{-1} C_1^* D_2 $. Since $U$ is isometry,
		$$
		U^*U = \left[ \begin{matrix}
			C_1^*C_1 + C_2^*C_2 & C_1^*D_1 & C_1^*D_2 + C_2^*D_3 \\
			D_1^*C_1 & B_1^*B_1 + D_1^*D_1 & D_1^*D_2 \\
			D_2^*C_1 + D_3^*C_2 & D_2^*D_1 & D_2^*D_2 + D_3^*D_3
		\end{matrix} \right] 
		= \begin{bmatrix}
			I_{\cE} & 0 & 0 \\
			0 & I_{\mathcal{H}_1} & 0 \\
			0 & 0 & I_{\mathcal{H}_2}
		\end{bmatrix},
		$$
		which gives
		\begin{align*}
			&C_1^*D_1=0=D_1^*D_2,\quad C_1^*D_2 + C_2^*D_3= 0,\quad B_1^*B_1 + D_1^*D_1 = I_{\mathcal{H}_1},\quad
			D_2^*D_2 + D_3^*D_3 = I_{\mathcal{H}_2}\text{ and }
		\end{align*}
		$$ C_1^*C_1 + C_2^*C_2 = I_{\cE}\,\,\implies A^2 + C_2^*C_2 = I_{\cE}.$$
		Using the above relations and \cref{cond_1}, we have
		\begin{align*}
			B_2^*B_2 + D_3^*D_3 &= D_2^* C_1 A^{-2} C_1^* D_2 + D_3^*D_3\\
			&= D_2^* D_2 + D_3^* D_3 \\
			& = I_{\mathcal{H}_2},
		\end{align*}
		Thus
		$$U_1^*U_1 = \begin{bmatrix}
			A^{-1} C_1^*C_1 A^{-1} & A^{-1} C_1^*D_1 \\
			D_1^*C_1 A^{-1} & B_1^*B_1 +D_1^*D_1 
		\end{bmatrix}
		= \begin{bmatrix}
			I_{\cE} & 0 \\
			0 & I_{\mathcal{H}_1}
		\end{bmatrix}
		$$
		and 
		$$
		U_2^*U_2 = \begin{bmatrix}
			A^2 + C_2^*C_2 & A^* B_2 + C_2^*D_3 \\
			B_2^ *A + D_3^*C_2 & B_2^*B_2 + D_3^*D_3
		\end{bmatrix} 
		= \begin{bmatrix}
			I_{\cE} & 0 \\
			0 & I_{\mathcal{H}_2}
		\end{bmatrix}.
		$$
		As $U_1$ and $U_2$ are isometries, the operators 
		$$
		\tilde{U_1} := \left[\begin{array}{c| c c}
			0 & B_1 & 0 \\
			\hline 
			C_1 A^{-1} & D_1 & 0 \\
			0 & 0 & I_{\mathcal{H}_2}
		\end{array} \right] \,\,\, \text{and} \,\, 
		\tilde{U_2}:=
		\left[\begin{array}{c| c c}
			A & 0 & B_2 \\
			\hline 
			0 & I_{\mathcal{H}_1} & 0 \\
			C_2 & 0 & D_3 
		\end{array}\right]\quad
		$$
		on $ \cE \oplus (\mathcal{H}_1 \oplus \mathcal{H}_2)$ are isometries. Using the relation $C_1 A^{-1} B_2= D_2$, we get $\tilde{U_1}\tilde{U_2}=U$.
		Now,
		\begin{align*}
			\theta(x)&= [B_1,0] \begin{bmatrix}
				\rho_1(E(x)) & 0 \\
				0 & \rho_2(E(x))
			\end{bmatrix} \left( \begin{bmatrix}
				I_{\mathcal{H}_1} & 0 \\
				0 & I_{\mathcal{H}_2}
			\end{bmatrix}- \begin{bmatrix}
				D_1 & C_1 A^{-1}B_2 \\
				0 & D_3
			\end{bmatrix} \begin{bmatrix}
				\rho_1(E(x)) & 0 \\
				0 & \rho_2(E(x))
			\end{bmatrix}\right)^{-1} \begin{bmatrix}
				C_1\\
				C_2
			\end{bmatrix}\\
			&= [B_1 \rho_1(E(x)), 0 ] \begin{bmatrix}
				I_{\mathcal{H}_1} - D_1\rho_1(E(x)) & - C_1 A^{-1} B_2\rho_2(E(x)) \\
				0 & I_{\mathcal{H}_2}- D_3 \rho_2(E(x))
			\end{bmatrix}^{-1} \begin{bmatrix}
				C_1\\
				C_2
			\end{bmatrix}\\
			&=\psi_1(x) \psi_2(x)
		\end{align*}
		where, $$
		\psi_1(x)= B_1\rho_1(E(x))\left( I_{\mathcal{H}_1}- D_1\rho_1(E(x)) \right)^{-1} C_1 A^{-1}
		$$
		and $$
		\psi_2(x)=A + B_2\rho_2(E(x))\left( I_{\mathcal{H}_2}- D_3\rho_2(E(x)) \right)^{-1}C_2.
		$$

		Clearly, $\psi_1(x_0) =0$ and $\psi_2(x_0) = A $. And also, $U_1$ and $U_2$ are isometric colligations for $\psi_1$ and $\psi_2$, respectively. Therefore by \cref{Real-thm}, $\psi_1,\psi_2 \in \cS\cA_{\Psi}(\cE)$.
	\end{proof}
	
	Now we shall consider the case when $\psi_j(x_0)=0$ for $j=1,2$. The proof of the following theorem is more or less similar to the previous one. Hence, we omit the proof.
	\begin{thm}\label{Main2}
		Let $\theta \in \cS\cA_{\Psi}(\cE)$ with $\theta(x_0)=0$.Then, there exists $\psi_1,\psi_2 \in \cS\cA_{\Psi}(\cE)$ such that $\theta = \psi_1 \psi_2$ and $\psi_1(x_0)=0=\psi_2(x_0) $ if and only if there exist Hilbert spaces $\mathcal{H}_1,\mathcal{H}_2$, a reducible unital $*-$representation $\rho: C_b(\Psi) \rightarrow \cB\left(  \mathcal{H}_1 \oplus \mathcal{H}_2 \right) $ and an isometric colligation $$ 
		U= \left[ \begin{array}{c| c c}
			0 & B_1 & 0 \\
			\hline
			0 & D_1 & D_2 \\
			C_2 & 0 & D_3
		\end{array}\right] : \cE \oplus \left( \mathcal{H}_1 \oplus \mathcal{H}_2 \right) \to \cE \oplus \left( \mathcal{H}_1 \oplus \mathcal{H}_2 \right)
		$$
		such that $$ 
		\theta(x)= [B_1,0]\rho(E(x))\left( I_{\mathcal{H}_1 \oplus \mathcal{H}_2} - \left[\begin{smallmatrix}
			D_1 & D_2 \\
			0 & D_3
		\end{smallmatrix}\right]\rho(E(x)) \right)^{-1} \left[ \begin{smallmatrix}
			0\\
			C_2
		\end{smallmatrix}\right]
		$$
		with $L^*D_1 =0 $ and $D_2 = LY$ for some $Y\in \mathcal{B}(\mathcal{H}_2,\cE)$ and isometry $L$ on $\cH_2$.
	\end{thm}In \cref{Main1} and \cref{Main2}, we assumed that $\theta(x_0)=0$. The following theorem characterizes the factorization of $\theta$ with out any assumption on $\theta(x_0)$.
	\begin{thm} \label{Main3}
		Let $\theta \in \cS\cA_{\Psi}(\cE)$ with $\theta(x_0) = A$. Then $\theta = \psi_1 \psi_2$ for some $\psi_1,\psi_2 \in \cS\cA_{\Psi}(\cE)$ if and only if  there exist Hilbert spaces $\mathcal{H}_1,\mathcal{H}_2$, a reducible unital $*-$representation $\rho: C_b(\Psi)\rightarrow \cB(\mathcal{H}_1 \oplus \mathcal{H}_2)$ and an isometric colligation $$
		U= \left[ \begin{array}{c| c c}
			A & B_1 & B_2 \\
			\hline
			C_1 & D_1 & D_2 \\
			C_2 & 0 & D_3
		\end{array}\right] :  \cE \oplus \left( \mathcal{H}_1 \oplus \mathcal{H}_2 \right) \to \cE \oplus \left( \mathcal{H}_1 \oplus \mathcal{H}_2 \right)$$ 
		such that $A=A_1 A_2$ and there exist operators $X_1$ and $Y_2$ satisfying
		 \begin{align} \label{cond-2} 
			 B_2 = A_1 Y_2, \ C_1= X_1 A_2, \ D_2=X_1 Y_2,\ A_1^* A_1 + X_1^* X_1 = I,
		\end{align}
		\begin{align}\label{cond-2.0}
			 A_2^* \text{ is injective on } \text{Range}(A_1^* B_1 + X_1^* D_1), 
		\end{align}		for some $ A_1,A_2 \in \cB(\cE)$ and
		\begin{align}\label{theta}
			\theta(x)= A + [B_1,B_2] \rho(E(x))\left( I_{\mathcal{H}_1 \oplus \mathcal{H}_2} - \begin{bmatrix}
				D_1 & D_2\\
				0 & D_3
			\end{bmatrix} \rho(E(x)) \right)^{-1} \begin{bmatrix}
				C_1 \\
				C_2
			\end{bmatrix}. 
		\end{align} 
	\end{thm}
	
	\begin{proof}
		Suppose $\psi_1,\psi_2 \in \cS\cA_{\Psi}(\cE) $ such that $\theta = \psi_1 \psi_2$. Then by \cref{Real-thm}, there exist Hilbert spaces $\mathcal{H}_1,\mathcal{H}_2$ and a unital $*-$representations $\rho_j: C_b(\Psi) \rightarrow \cB(\cH_j)$ for $j=1,2$, acting on $\mathcal{H}_1$ and $\mathcal{H}_2$ respectively with isometric colligations $U_1,U_2$ as follows:
		$$ U_j= \begin{bmatrix}
			A_j & B_j \\
			C_j & D_j
		\end{bmatrix}:\cE \oplus \mathcal{H}_j \to \cE \oplus \mathcal{H}_j
		$$ such that $$ 
		\psi_j(x)= A_j + B_j\rho_j(E(x))\left( I_{\mathcal{H}_j} - D_j\rho_j(E(x))\right)^{-1}C_j
		$$ for $j=1,2.$
		Define $$
		U = \left[\begin{array}{c| c c}
			A_1 & B_1 & 0\\
			\hline
			C_1 & D_1 & 0 \\
			0 & 0 & I_{\mathcal{H}_2}
		\end{array}\right] \left[\begin{array}{c| c c}
			A_2 & 0 & B_2 \\
			\hline
			0 & I_{\mathcal{H}_1} & 0  \\
			C_2 & 0 &  D_2 \\
		\end{array}\right] = \left[\begin{array}{c|c c}
			A_1A_2 & B_1 & A_1B_2 \\
			\hline
			C_1 A_2 & D_1 & C_1B_2 \\
			C_2 & 0 & D_2
		\end{array}\right]
		$$
		and $$
		\rho(g) = \begin{bmatrix}
			\rho_1(g) & 0 \\
			0 & \rho_2(g)
		\end{bmatrix} : \cH_1 \oplus \cH_2 \to  \cH_1 \oplus \cH_2 $$ for each $g \in C_b(\Psi)$.
		Clearly, $U$ is an isometry and $\rho$ is a representation of $C_b(\Psi)$. A straightforward computation, as we did earlier, gives that the transfer function realization corresponding to the isometry $U$ and the representation $\rho$ is $\theta$.
		Also, it is easy to check that the isometry $U$ satisfies condition \eqref{cond-2} and \eqref{cond-2.0}.
		
		Conversely, suppose that there exist Hilbert spaces $\mathcal{H}_1,\mathcal{H}_2$, a reducible unital $*-$representation $\rho: C_b(\Psi)\rightarrow \cB(\mathcal{H}_1 \oplus \mathcal{H}_2)$ and $g\in C_b(\Psi)$  and an isometric colligation $$
		U= \left[ \begin{array}{c| c c}
			A & B_1 & B_2 \\
			\hline
			C_1 & D_1 & D_2\\
			C_2 & 0 & D_3
		\end{array}\right] : \cE \oplus \left( \mathcal{H}_1 \oplus \mathcal{H}_2 \right) \to \cE \oplus \left( \mathcal{H}_1 \oplus \mathcal{H}_2 \right)
		$$ satisfying \eqref{cond-2}and \eqref{cond-2.0} such that $\theta(x)$ is as in equation \eqref{theta}.
		By our assumption, we can write $$
		\rho(g) = \left[\begin{smallmatrix}
			\rho_1(g)|_{\cH_1} & 0 \\
			0 & \rho_2(g)|_{\cH_2}
		\end{smallmatrix}\right] $$ on $\mathcal{H}_1 \oplus \mathcal{H}_2$ for $g \in C_b(\Psi)$. Set
		$$
		U_1 = \begin{bmatrix}
			A_1 & B_1 \\
			X_1  & D_1 
		\end{bmatrix} \quad \text{and} \quad
		U_2 = \begin{bmatrix}
			A_2 & Y_2 \\
			C_2 & D_3 
		\end{bmatrix}.
		$$
		Now, we shall prove that $U_1$ and $U_2$ are isometries. Since $U$ is an isometry,  we have
		\begin{align}\label{eq-1}
			A^*A + C_1^*C_1 + C_2^*C_2 = I_{\cE}, \quad  A^* B_1 + C_1^*D_1 = 0 , \quad  A^*B_2 + C_1^*D_2+ C_2^*D_3 =0 
		\end{align}
		\begin{align}\label{eq-2}
			B_1^*B_1 + D_1^*D_1=I_{\mathcal{H}_1},\quad B_1^*B_2 + D_1^*D_2=0 ,\quad B_2^*B_2+ D_2^*D_2 + D_3^*D_3 = I_{\mathcal{H}_2}.
		\end{align}
		A simple computation gives that,$$
		U_1^*U_1 = \begin{bmatrix}
			A_1^*A_1 + X_1^* X_1  & A_1^*B_1 + X_1^*D_1 \\
			B_1^* A_1 + D_1^*X_1 &  B_1^*B_1 + D_1^*D_1
		\end{bmatrix}$$
		and 
		$$
		U_2^*U_2 = \begin{bmatrix}
			A_2^*A_2 + C_2^*C_2 & A_2^* Y_2+ C_2^*D_3 \\
			Y_2^* A_2 + D_3^*C_2 & Y_2^* Y_2 + D_3^*D_3
		\end{bmatrix}.
		$$

By our assumption, the $(1,1)$ entry in $U_1^*U_1$ is the same as $I_{\cE}$. From the second relation in \eqref{eq-1} can be written as
$A_2^*(A_1^*B_1 + X_1^* D_1)=0$ which implies $A_1^*B_1 + X_1^* D_1=0$ since $A_2^*$ is injective on $\text{Range}(A_1^*B_1 + X_1^* D_1)$. So, the $(1,2)$ entry and $(2,1)$ entry of $U_1^*U_1$ are $0$. Also, from equation \eqref{eq-2}, the $(2,2)$ entry is $I_{\cH_1}$. 

Now, from the first relation in $\eqref{eq-1}$, we have $$A_2^*(A_1^*A_1 + X_1^* X_1)A_2 + C_2^*C_2 = I $$ as $C_1 = X_1 A_2 $ and therefore,
$A_2^*A_2 + C_2^* C_2 = I$ by using \eqref{cond-2}. So, the $(1,1)$ entry of $U_2^*U_2$ is $I_{\cE}$. Equation \eqref{eq-1} together with equation \eqref{cond-2} gives the following
\begin{align*}
	0 &= A_2^*A_1^* A_1 Y_2 + A_2^* X_1^* X_1 Y_2 + C_2^*D_3 \\&= A_2^*(A_1^* A_1 + X_1^* X_1) Y_2 + C_2^*D_3\\&= A_2^* Y_2 + C_2^*D_3 .
\end{align*}
		
 Hence the $(1,2)$ and $(2,1)$ entries of $U_2^*U_2$ are zero. Further, the last equation in \eqref{eq-2} and the relations in \eqref{cond-2} implies that $$I = Y_2^*(A_1^*A_1 + X_1^* X_1) Y_2 + D_3^* D_3 = Y_2^* Y_2 + D_3^* D_3.$$ Therefore, the $(2,2)$ entry of $U_2^*U_2$ is $I_{\mathcal{H}_2}$. 	
		
Hence $U_1$ and $U_2$ are both isometries, which in turn implies that the operators
		$$ \tilde{U_1}= \left[\begin{array}{c| c c}
			A_1 & B_1 & 0\\
			\hline
			X_1 & D_1 & 0 \\
			0& 0 & I_{\mathcal{H}_2} 
		\end{array}\right] \quad \text{and}\quad \tilde{U_2}=\left[\begin{array}{c| c c}
			A_2 & 0 & Y_2\\
			\hline
			0 & I_{\mathcal{H}_1}&  0 \\
			C_2& 0 & D_3 
		\end{array}\right].
		$$
		are isometries on $\cE \oplus \left( \mathcal{H}_1 \oplus \mathcal{H}_2 \right)$.
		Let $\psi_1$ and $\psi_2$ be the transfer function realizations corresponding to the pairs ($\tilde{U_1}$ , $\rho_1 $) and  ($\tilde{U_2}$, $  \rho_2 $) respectively, where
		$$\rho_1 (g)= 
		\begin{bmatrix}
		\rho(g)|_{\cH_1} & 0\\ 0 & I_{\cH_2}
		\end{bmatrix}
		 \ \text{and}\ \rho_2(g)=\begin{bmatrix}
		 	I_{\cH_1} & 0\\ 0 & \rho(g)|_{\cH_2}
		 \end{bmatrix}
		 $$ 
		 for $g\in C_b(\Psi)$. A similar computation as in the proof of \cref{Main1} gives that 
		$$
		\theta(x)= \psi_1(x) \psi_2(x) \quad \text{for all } x \in X.
		$$
		This completes the proof.
	\end{proof}
	
	Note that in the case of $\bD^d$ the reducibility condition on the representation is automatically satisfied as the co-ordinate functions are the test functions.

	\textbf{Examples:} There are some other domains: the symmetrized bidisc \cite{Tirtha-Hari-JFA} and multi-connected domains \cite{BH, DM}, where the collections of test functions are known. Moreover, in these examples, the test functions are holomorphic. However, in both these cases, the number of such functions is uncountable. For the sake of brevity, we refrain from writing them explicitly. Nevertheless, we would like to emphasize that the test functions in these domains are certain inner functions. For more details on inner functions, see \cite{Bhowmik-Kumar} for the symmetrized bidisc and \cite{Fisher} for multi-connected domains.

	\begin{remark}
		It is worth noting that we have the test functions discussed above for the bidisc, the symmetrized bidisc, and the annulus, such that the $\Psi-$Schur-Agler class coincides with the Schur class functions. Consequently, we establish the factorization for the Schur class functions. 
	\end{remark}
	
\noindent\textbf{A comment on the extreme points:} Suppose $\Omega$ is either $\mathbb{D}$ or $\mathbb{D}^2$. If $\theta $ is a scalar-valued Schur class function on $\Omega$ and $\theta(0)=0$ such that there exists an isometric colligation operator $U$ satisfying the conditions in \cref{Main3}, then $\theta$ can be factorized as $\psi_1 \psi_2$ where $\psi_1, \psi_2 $ are Schur class functions on $\Omega$. So, Forelli's theorem implies that the Cayley transform of $\theta$ is not an extreme point of the normalized Herglotz class functions $\cN(\Omega, 0)$ provided $\psi_1$ and $\psi_2$ are non-constant. In general, determining all the extreme points of $\cN(\bD^2, 0)$ is a very difficult problem and it is still open. See for example \cite{Knese}. So, our findings assist in excluding certain normalized Herglotz class functions from being considered as extreme points.

We shall end this article with the following remark:
\begin{remark}
Suppose we are in the classical setup. Then note that our factorization results apply to any Schur class function. The set of all inner functions is a subclass of the Schur class functions. Recently Curto, Hwang, and Lee ( \cite{Curto2, Curto1}) have  studied operator-valued functions, with a particular focus on operator-valued inner functions and their factorization. It's natural to ask if there is any connection between these two sets of factorization results. However, exploring this connection requires a better understanding of factorization results for inner functions within our setup. We defer this inquiry to future research.
\end{remark}

\textbf{Funding:} The first author's work is supported by the Prime Minister's Research Fellowship PMRF-21-1274.03 and the second author is partially supported by a PIMS postdoctoral fellowship.

\textbf{Acknowledgement:} The authors  are thankful to their research supervisor Prof. Tirthankar Bhattacharyya for some discussions.  The authors want to thank the anonymous referees for their helpful comments and suggestions that truly improved the paper.

\end{document}